\documentclass[12pt,twoside,reqno]{amsart}

\usepackage[OT1]{fontenc}
\usepackage{type1cm}
\usepackage[english]{babel}
\usepackage{amsthm}
\usepackage{epsfig}
\usepackage{graphicx,epsfig,subfigure}

\numberwithin{equation}{section}

\theoremstyle{definition}
\newtheorem{thm}{Theorem}[section]
\theoremstyle{definition}

\theoremstyle{definition}

\theoremstyle{definition}

\theoremstyle{definition}

\theoremstyle{remark}
\newtheorem{rem}[thm]{Remark}

\newcommand{\R}{\mathbb{R}}
\newcommand{\Rn}{\mathbb{R}^{n}}

\newcommand{\N}{\mathbb{N}}

\begin{document}
\title{Singular Integrals on Sierpinski Gaskets }
\author{Vasilis Chousionis}

\thanks{The author is supported by the Finnish Graduate School in
Mathematical Analysis.} 

\begin{abstract}We construct a class of singular
integral operators associated with homogeneous Calder\'{o}n-Zygmund
standard kernels on $d$-dimensional, $d <1$, Sierpinski gaskets
$E_d$. These operators are bounded in $L^2(\mu_d)$ and their
principal values diverge $\mu_d$ almost everywhere, where $\mu_d$ is
the natural (d-dimensional) measure on $E_d$.
\end{abstract}

\maketitle
\section{Introduction}
Given a Radon measure $\mu$ on $\Rn$ and a continuously
differentiable kernel
$K:\Rn\times\Rn\setminus\{(x,y):x=y\}\rightarrow \R$ that satisfies
the antisymmetry condition
\begin{equation*}K(x,y)=-K(y,x)\text { for
}x,y\in\Rn\text{, }x\neq y,
\end{equation*}the singular integral operator $T$ associated with $K$ and $\mu$
is formally given by
\begin{equation*}
T(f)(x)=\int K(x,y)f(y)d\mu y.
\end{equation*}
Notice that the above integral does not usually exist when $x\in
\textmd{spt} \mu$. The truncated singular integral operators
$T_{\varepsilon}$, $\varepsilon >0$;
\begin{equation*}
T_{\varepsilon}(f)(x)=\int_{|x-y|>\varepsilon} K(x,y)f(y)d\mu y,
\end{equation*}
are considered in order to overcome this obstacle. In the same vein
one considers the maximal operator $T^*$
\begin{equation*}
T^*(f)(x)=\underset{\varepsilon > 0}{\sup }|T_{\varepsilon}(f)(x)|
\end{equation*}
and the principal values of $T(f)$ at every $x \in \Rn$ which, if
they exist, are given by
\begin{equation*}
\textmd{p.v.}T(f)(x)=\underset{\varepsilon \rightarrow 0}{\lim
}T_{\varepsilon}(f)(x).
\end{equation*}
The singular integral operator $T$ associated with $\mu$ and $K$ is
said to be bounded in $L^2(\mu)$ if there exists some constant $C>0$
such that for $f\in L^2(\mu)$ and $\varepsilon >0$
\begin{equation*}
\int {|T_{\varepsilon}(f)|}^2 d\mu \leq C\int {|f|}^2d\mu.
\end{equation*}

The natural question as to whether the $L^2(\mu)$-boundedness of the
operator $T$ forces its principal values to exist $\mu$ almost
everywhere has been considered in many papers (see e.g. \cite{MM},
 \cite{MMV}, \cite{T}, \cite{D}, \cite {H}, \cite{MV} and \cite {Pr}).
Even when $\mu$ is an $m$-dimensional
Ahlfors-David (AD) regular measure in $\Rn$:
\begin{equation*}
 C^{-1}r^{m}\leq \mu (B(x,r))\leq Cr^{m} \text{ for }x \in
 \textmd{spt}\mu, 0 <r<\textmd{diam}(\textmd{spt}\mu),
\end{equation*}
and $K$ is any of the coordinate Riesz kernels:
\begin{equation*}
R_i^{m} (x,y)=\frac {x_i-y_i}{|x-y|^{m+1}}\text{ for } i=1,...,n
\end{equation*}
the question remains open for $m>1$.

When $m=1$, or equivalently in the case of the Cauchy transform, the
above question has a positive answer by the results of Mattila,
Melnikov and Verdera (see \cite{MM} and \cite{MMV}). Later on, in
\cite{T}, Tolsa improved the afore mentioned results by removing the
Ahlfors-David regularity assumption.

In different settings the answer to the above question can be
negative. Let $C$ be the 1-dimensional four corners Cantor set 
and $\mu$ its natural (1-dimensional Hausdorff) measure. David in
\cite{D}, constructed Calder\'{o}n-Zygmund standard, or simply CZ
standard, kernels that define operators bounded in $L^2(\mu)$ whose
principal values fail to exist $\mu$ almost everywhere. Although
David's kernels can be chosen odd or even are not homogeneous of
degree -1.

In this note we consider classical plane Sierpinski gaskets of
Hausdorff dimension $d$, $0<d<1$. For each of these $d$-AD regular
sets $E_d$, we find families of CZ standard, smooth, and
antisymmetric kernels of the form
\begin{equation}
\label{ke} K(x,y)=\frac {\Omega((x-y)/|x-y|)}{h(|x-y|)}
\end{equation}
where $h$ is some increasing $C^{\infty}$ function satisfying the
homogeneity condition $$h(r)\simeq r^d$$ for
$0<r<\textmd{diam}(E_d)$ and $\Omega$ is odd on the unit circle
$S^1$. If $\mu_d$ is the restriction of the $d$-dimensional
Hausdorff measure on $E_d$, these kernels define
 singular integral operators bounded in $L^2(\mu_d)$
 whose principal values diverge ${\mu}_d$ almost everywhere. The proof is
based on the $T(1)$-theorem of David and Journ\'{e}, proved in
\cite{dj}, and the symmetry properties of Sierpinski gaskets
allowing heavy cancelations.
\begin{rem}
Unfortunately if in the above kernels we replace the function $h(r)$
by $r^d$, where in this case the kernels would be $d$-homogeneous in
the classical sense, we cannot say if the corresponding operators
are bounded (or not) in $L^2$. However their principal values
diverge ${\mu}_d$ almost everywhere, as the proof of Section 4 goes
through with no changes.
\end{rem}
\begin{rem}
Modified slightly, the proof can be applied to many other symmetric
self similar sets, e.g. the four corners Cantor sets with Hausdorff
dimension less than $1$ or the self similar sets discussed in
\cite{D}. The dimensional restriction is essential for the proof and
it is not known to us if there exist CZ standard kernels, of the
same form as in (\ref{ke}), satisfying the homogeneity condition $h(r)\approx r$,
that define singular integral operators bounded in $L^2$ but whose
principal values diverge almost everywhere.
\end{rem}

\section{Notation and Setting}

Let $\lambda \in (0,1/3)$ and consider the following three
similitudes (depending on $\lambda $) $s_{1}^{\lambda
},s_{2}^{\lambda },s_{3}^{\lambda }:\mathbb{R}^{2}\rightarrow
\mathbb{R}^{2}$
\begin{itemize}
\item $s_{1}^{\lambda }(x,y) =\lambda (x,y)$
\item $s_{2}^{\lambda }(x,y) =\lambda (x,y)+(1-\lambda ,0)$
\item $s_{3}^{\lambda }(x,y) =\lambda
(x,y)+(\frac{1-\lambda}{2},\frac{\sqrt{3}}{2}(1-\lambda)).$
\end{itemize}

Let $I=\{1,2,3\}$ and $I^{\ast }=\underset{n\geq 1}{\bigcup }I^{n}$. The set $%
I^{\ast }$ can be partially ordered in the following way, for
$\alpha ,\beta \in I^{\ast }$,
\begin{equation*}
\alpha \prec \beta \Leftrightarrow \alpha \in I^{n},\beta \in I^{k}\text{ }%
k\geq n\text{ and }\beta \lfloor n=\alpha.
\end{equation*}
Where $\beta \lfloor n$ denotes the restriction of $\beta$ in its
first $n$ coordinates. For $\alpha \in I^{n}$, say $\alpha
=(i_{1},...,i_{n})$, define $s_{\alpha }^\lambda$
$:\mathbb{R}^{2}\rightarrow \mathbb{R}^{2}$ through iteration
\begin{equation*}
s_{\alpha }^{\lambda }=s_{i_{1}}^{\lambda }\circ s_{i_{2}}^{\lambda
}\circ ...\circ s_{i_{n}}^{\lambda }.
\end{equation*}
Let $A$ be the equilateral triangle with vertices $(0,0),(1,0),(1/2,\sqrt{%
3}/2)$. Denote $s_{\alpha }^{\lambda }(A)=S_{\alpha }^{\lambda }$,
$I^{0}=\{0\}$ and $s_{0}^{\lambda }=id$. The limit set of the
iteration
\begin{equation*}
E_{\lambda }=\underset{j\geq 0}{\bigcap }\underset{\alpha \in
I^{j}}{\bigcup }S_{\alpha }^{\lambda }
\end{equation*}
is self similar and in fact it is an $\lambda$-Sierpinski triangle
with Hausdorff dimension
\begin{equation*}
d_{\lambda }=\dim _{\mathcal{H}}E_{\lambda }=-\frac{\log 3}{\log
\lambda }.
\end{equation*}
Notice that for $\lambda \in (0,1/3)$, $d_{\lambda }\in (0,1)$. As a
general
property of self similar sets the measures $\mu _{\lambda }=$ $\mathcal{H}%
^{d_{\lambda }}\lfloor E_{\lambda }$ are $d_\lambda$-AD regular.
Hence there exists a constant $C_{\lambda }$, depending only on
$\lambda $, such that for $x\in E_{\lambda }$ and $0<r\leq1$,
\begin{equation*}
C_{\lambda }^{-1}r^{d_{\lambda }}\leq \mu _{\lambda }(B(x,r))\leq
C_{\lambda }r^{d_{\lambda }}.
\end{equation*}

The spaces $(E_{\lambda },\rho ,\mu _{\lambda })$, where $\rho$ is
the usual Euclidean metric, are simple examples of spaces of
homogeneous type (See \cite{Ch} for definition). We want to find
Calder\'{o}n-Zygmund standard  kernels on $E_{\lambda }\times
E_{\lambda }\setminus\{(x,y):x=
y\}$ that define bounded singular integral operators on $%
L^{2}(\mu _\lambda )$. In that direction, for $\lambda \in (0,1/3)$,
we need to define two auxiliary families of functions.

\emph{The functions $\Omega _{\lambda }$}: For any pair $(x,y)\in \mathbb{R}%
^{2}\times \mathbb{R}^{2}$, $x\neq y$, denote by $\theta _{(x,y)}\in
\lbrack 0,2\pi )$ the angle formed by the vectors $y-x$ and
$e_{1}=(1,0)$.
 For every $\lambda \in (0,1/3)$ there exists some positive number
$\varepsilon _\lambda $ such that

\begin{enumerate}
\item  For all $x,y\in E_{\lambda }$, $x\neq y$
\begin{equation*}
\theta _{(x,y)}\in (\frac{k\pi }{3}-\varepsilon _{\lambda },\frac{k\pi }{3}%
+\varepsilon _{\lambda })\text{ for some }k\in \{0,1,..,5\}.
\end{equation*}

\item  The intervals $(\dfrac{k\pi }{3}-\varepsilon _{\lambda },\dfrac{k\pi
}{3}+\varepsilon _{\lambda })$ are disjoint for $k\in \{0,1,..,5\}$.

\item For any $n\in \N=\{1,2,...\}$ and $\alpha ,\beta,\gamma \in I^n$, $\alpha \neq \beta \neq \gamma $, such
that $\alpha \lfloor n-1=\beta \lfloor n-1=\gamma \lfloor n-1$:
\begin{enumerate}
\item If $x \in S_{\alpha }^{\lambda }$, $y \in S_{\beta }^{\lambda
}$, $z \in S_{\gamma }^{\lambda }$ and $\theta _{(x,y)}\in
(\dfrac{k\pi }{3}-\varepsilon _{\lambda },\dfrac{k\pi
}{3}+\varepsilon _{\lambda })$, for some $k\in \{0,1,...,5\}$, then
$\theta _{(x,z)}\in (\dfrac{m\pi }{3}-\varepsilon _{\lambda
},\dfrac{m\pi }{3}+\varepsilon _{\lambda })$ for $m=(k+1)
\textmd{mod}6$ or $m=(k-1)\textmd{mod}6$.

\item If $x,z\in S_{\alpha }^{\lambda }$, $y\in S_{\beta }^{\lambda
}$ and
 $\theta _{(x,y)}\in (\dfrac{k\pi }{3}-\varepsilon _{\lambda },\dfrac{k\pi }{3}%
+\varepsilon _{\lambda })$ then $%
\theta _{(z,y)}\in (\dfrac{k\pi }{3}-\varepsilon _{\lambda },\dfrac{k\pi }{3}%
+\varepsilon _{\lambda })$ as well.
\end{enumerate}
\end{enumerate}

Now we can define $C^{\infty}$ functions $\Omega _{\lambda }$ on $%
S^{1}$ satisfying
\begin{enumerate}

\item  $\Omega _{\lambda }(z)=(-1)^{k}$ for $\theta _{(z,0)}\in (\frac{k\pi }{
3}-\varepsilon _{\lambda },\frac{k\pi }{3}+\varepsilon _{\lambda
})$, $k\in \{0,1,...,5\}$,
\item $\Omega _{\lambda }(-z)=-\Omega
_{\lambda }(z)$  for every  $z\in S^1$.

\end{enumerate}

Observe that the second condition also implies
\begin{equation*}
\int_{S^1}\Omega_\lambda(z) d\sigma z=0
\end{equation*}
where $\sigma$ is the normalized surface measure on $S^1$.

\emph{The functions $h_{\lambda }$}: Fix some $\lambda \in (0,1/3)$,
and choose any function $h_\lambda:(0,\infty) \rightarrow \R$ with
the following properties,
\begin{enumerate}
\item  $h_{\lambda }$ is $C^{\infty }$,

\item  $h_{\lambda }$ is increasing,

\item  $h_{\lambda} \lfloor [(\frac{1}{\lambda}-2)\lambda ^{k},\lambda ^{k-1}]=\lambda ^{(k-1)d_{\lambda }}$
for every $k\in \mathbb{N}$.
\end{enumerate}
It follows that for $r\in (0,1]$, $h_{\lambda \text{ }}(r)\approx
r^{d_{\lambda }}$. In fact
\begin{equation*}
r^{d_{\lambda }}/C_{\lambda }\leq h_{\lambda \text{ }}(r)\leq
C_{\lambda }r^{d_{\lambda }}\text{ \ for }0<r\leq 1
\end{equation*}
where $C_{\lambda }=\lambda ^{-d_{\lambda }}$.

Hence we are able, using the above families, to define appropriate
kernels
\begin{equation*}
K_{\lambda }:E_{\lambda }\times E_{\lambda }\backslash
\{(x,y):x=y\}\rightarrow \mathbb{R}
\end{equation*}
as
\begin{equation*}
K_{\lambda }(x,y)=\frac{\Omega _{\lambda }((x-y)/\left| x-y\right| )}{%
h_{\lambda }(|x-y|)}.
\end{equation*}
For the kernels $K_{\lambda }$ there exists some constant $C$ such
that for all $x,y,z\in E_{\lambda }$, $x\neq y$, satisfying $\left|
x-z\right| <(1-2 \lambda)\left| x-y\right|$,

\begin{equation}
\label{k1}
\left| K_{\lambda }(x,y)\right| \leq \frac{C}{\left|
x-y\right| ^{d_{\lambda }}},
\end{equation}
\begin{equation}
\label{k2}
K_{\lambda }(x,y)-K_{\lambda }(z,y) =0
\end{equation}
Condition (\ref{k1}) follows immediately from the definition of
$K_{\lambda}$. To prove (\ref{k2}), let $k\in\mathbb{N}^{\ast
}=\{0,1,..\}$ be the largest natural number such that $x,y\in
S_{\alpha }^{\lambda }$ for some $ \alpha \in I^{k}$. Therefore
\begin{equation*}
x\in s_{i}^{\lambda }(S_{\alpha }^{\lambda })\text{ and }y\in
s_{j}^{\lambda }(S_{\alpha }^{\lambda })
\end{equation*}
for some $i,j \in I$, $i \neq j$. This implies that,
\begin{equation}
\label{1} (\frac{1}{\lambda }-2)\lambda ^{k+1}\leq\left| y-x\right|
\leq \lambda ^{k}.
\end{equation}
Since $\left| x-z\right| <(1-2 \lambda)\left| x-y\right|$ we get
\begin{equation*}
\left| x-z\right| <(1-2\lambda)\lambda ^{k}.
\end{equation*}
As
\begin{equation*}
d(s_{i}^{\lambda }(S_{\alpha }^{\lambda }),s_{q}^{\lambda
}(S_{\alpha }^{\lambda }))=(1-2\lambda)\lambda ^{k}\text{ for } q\in
I\text{, }q \neq i,
\end{equation*}
and
\begin{equation*}
S_{\alpha }^{\lambda }=\underset{p\in I}{\bigcup }s_{p}^{\lambda
}(S_{\alpha }^{\lambda }),
\end{equation*}
we deduce that $z\in s_{i}^{\lambda }(S_{\alpha }^{\lambda })$ and
\begin{equation}
\label{2} (\frac{1}{\lambda }-2)\lambda ^{k+1}\leq\left| y-z\right|
\leq \lambda ^{k}.
\end{equation}
Therefore as $x,z\in s_{i}^{\lambda }(S_{\alpha }^{\lambda })$ and
$y\in
s_{j}^{\lambda }(S_{\alpha }^{\lambda })$%
\begin{equation}
\label{3}
\theta (x,y),\theta (z,y)\in \left( m\frac{\pi }{3}-\varepsilon _{\lambda },m%
\frac{\pi }{3}+\varepsilon _{\lambda }\right)
\end{equation}
for some $m\in \{0,1,..,5\}$. From (\ref{1}), (\ref{2}), (\ref{3})
and the definition of $h_\lambda$ and $\Omega_\lambda$ we deduce
that
\begin{equation*}
h_{\lambda }(\left| x-y\right| )=h_{\lambda }(\left| z-y\right|
)=\lambda^{kd_{\lambda}},
\end{equation*}
and
\begin{equation*}
\Omega _{\lambda }\left( \frac{x-y}{\left| x-y\right| }\right)
=\Omega _{\lambda }\left( \frac{z-y}{\left| z-y\right| }\right).
\end{equation*}
Hence
\begin{equation*}
 K_{\lambda }(x,y)-K_{\lambda }(z,y) =0
\end{equation*}
and by antisymmetry
\begin{equation*}
 K_{\lambda }(y,x)-K_{\lambda }(y,z) =0.
\end{equation*}
It follows that the kernels $K_\lambda$ are  CZ standard, in fact
condition (\ref{k2}) is much stronger than the ones appearing in the
usual definitions of  CZ standard kernels (see e.g. \cite{Ch},
\cite{Db} and \cite{j}).

As stated before, we want to show that the kernels define singular
integral operators that are bounded in $L^{2}(\mu _{\lambda })$, and
examine their convergence properties.

\section{$L^{2}$ boundedness}

\begin{thm}
For all $\lambda \in (0,1/3)$ the maximal singular integral
operators $T_{\lambda }^{\ast}$,
\begin{equation*}
T_{\lambda }^{\ast}(f)(x)=\underset{\varepsilon > 0}{\sup }\left|
\int_{\left| x-y\right| > \varepsilon }K_{\lambda }(x,y)f\left(
y\right) d\mu _{\lambda }y\right|,
\end{equation*}
are bounded in $L^{2}(\mu _{\lambda })$.
\end{thm}

\begin{proof}The idea is to use the $T(1)$ theorem of David and Journ\'{e} in the context of \cite{D}. Start by defining $T_{\lambda}^{n}$, for $n\geq 1$, as
\begin{equation*}
T_{\lambda}^{n}(f)(x)=\int_{\left| x-y\right| > \lambda
^{n}}K_{\lambda }(x,y)f\left( y\right) d\mu _{\lambda }y.
\end{equation*}
We want to show that $T_{\lambda}^{n}(\mathbf{1)=}0$ for all $n\in \mathbb{N%
}$, by induction.

For $n=1$: Let $x\in S_{i}^{\lambda} \cap E_{\lambda } $ for some
$i\in I$. If $j,k\in I\setminus\{i\}$, $j \neq k$, we get
\begin{eqnarray*}
T_{\lambda }^{1}(\mathbf{1})(x) &=&\int_{\left| x-y\right| > \lambda
}K_{\lambda }(x,y)d\mu _{\lambda }y \\
&=&\int_{S_{j}^{\lambda }\cup S_{k}^{\lambda
}}K_{\lambda }(x,y)d\mu _{\lambda }y \\
&=&\int_{S_{j}^{\lambda }}\frac{\Omega _{\lambda }((x-y)/\left|
x-y\right|
)}{h_{\lambda }(|x-y|)}d\mu _{\lambda }y+\int_{S_{k}^{\lambda }}\frac{%
\Omega _{\lambda }((x-y)/\left| x-y\right| )}{h_{\lambda
}(|x-y|)}d\mu _{\lambda }y.
\end{eqnarray*}
Furthermore there exists some $m\in \{0,1,..,5\}$ such that for
$y\in $ $S_{j}^{\lambda }$,
\begin{equation*}
\Omega _{\lambda }\left( \frac{x-y}{\left| x-y\right| }\right)
=(-1)^{m}
\end{equation*}
and for $y\in $ $S_{k}^{\lambda}$,
\begin{equation*}
\Omega _{\lambda }\left( \frac{x-y}{\left| x-y\right| }\right)
=(-1)^{m+1}.
\end{equation*}
Hence
\begin{equation*}
T_{\lambda }^{1}(\mathbf{1})(x)=(-1)^{m}\int_{S_{j}^{\lambda }}\frac{1}{%
h_{\lambda }(|x-y|)}d\mu _{\lambda }y+(-1)^{m+1}\int_{S_{k}^{\lambda }}\frac{%
1}{h_{\lambda }(|x-y|)}d\mu _{\lambda }y
\end{equation*}
But for $y\in S_{j}^{\lambda }\cup S_{k}^{\lambda }$ \ we have that $%
1-2\lambda \leq$ $\left| x-y\right| \leq 1$ and consequently
$h_\lambda(\left| x-y\right| )=1$. Thus
\begin{equation*}
T_{\lambda }^{1}(\mathbf{1})(x)=(-1)^{m}\mu _{\lambda
}(S_{j}^{\lambda })+(-1)^{m+1}\mu _{\lambda }(S_{k}^{\lambda })=0.
\end{equation*}

Suppose that $T_{\lambda }^{n}(\mathbf{1})=0$ and
let some $x\in E_{\lambda }$. We want to show that $T_{\lambda }^{n+1}(\mathbf{1})(x)=0$%
. Let $x\in S_{\alpha }^{\lambda }$ for some $\alpha
=(i_{1},i_{2},..,i_{n},i_{n+1})\in I^{n+1}$. If $\beta
=(i_{1},i_{2},..i_{n},j)$
 and $\gamma =(i_{1},i_{2},..i_{n},k)$ for $j,k\in
I\setminus\{i_{n+1}\}$, $j \neq k$,
\begin{eqnarray*}
T_{\lambda }^{n+1}(\mathbf{1})(x) &=&\int_{\left| x-y\right|
>\lambda
^{n+1}}K_{\lambda }(x,y)d\mu _{\lambda }y \\
&=&\int_{\left| x-y\right| > \lambda ^{n}}K_{\lambda }(x,y)d\mu
_{\lambda }y+\int_{S_{\beta }^{\lambda }}K_{\lambda }(x,y)d\mu
_{\lambda
}y+\int_{S_{\gamma }^{\lambda }}K_{\lambda }(x,y)d\mu _{\lambda }y \\
&=&\int_{S_{\beta }^{\lambda }}\frac{\Omega _{\lambda }((x-y)/\left|
x-y\right| )}{h_{\lambda }(|x-y|)}d\mu _{\lambda }y+\int_{S_{\gamma
}^{\lambda }}\frac{\Omega _{\lambda }((x-y)/\left| x-y\right| )}{%
h_{\lambda }(|x-y|)}d\mu _{\lambda }y,
\end{eqnarray*}
since by the induction hypothesis
\begin{equation*}
T_{\lambda }^{n}(\mathbf{1})(x)=\int_{\left| x-y\right| > \lambda
^{n}}K_{\lambda }(x,y)d\mu _{\lambda }y=0.
\end{equation*}
Using exactly the same argument as in the case for $n=1$
\begin{equation*}
\int_{S_{\beta }^{\lambda }}\frac{\Omega _{\lambda }((x-y)/\left|
x-y\right|
^{-1})}{h_{\lambda }(|x-y|)}d\mu _{\lambda }y+\int_{S_{\gamma }^{\lambda }}%
\frac{\Omega _{\lambda }((x-y)/\left| x-y\right| ^{-1})}{h_{\lambda }(|x-y|)}%
d\mu _{\lambda }y=0.
\end{equation*}
Therefore $T_{\lambda }^{n+1}(\mathbf{1})(x)=0$, completing the
induction. As $T_{\lambda }^{n}(\mathbf{1})=0$ for all $n\in
\mathbb{N}$ the same holds for their transposes.

Due to the structure of the spaces $(E_\lambda,\mu_\lambda,\rho)$
the proof of the $T(1)$ theorem in this setting is essentially the
same with the one appearing in \cite{Db}. As commented in \cite{Db}
and in \cite{D}, in order to be able to use the $T(1)$ theorem we
need some suitable decomposition of dyadic type, as in \cite{Ch} and
\cite{Db}, to replace the usual dyadic cubes in $\Rn$. In our
setting the required such family $\mathcal{R}$ consists of all the
triangles appearing in every step of the iteration process, i.e.
\begin{equation*}
\mathcal{R}_{k}^{\lambda }=\{S_{\alpha }^{\lambda }:\alpha \in
I^{k}\}\text{ for }k\in \mathbb{N}^{\ast }.
\end{equation*}
and
\begin{equation*}
\mathcal{R}=\{\mathcal{R}_{k}^{\lambda }:k\in \mathbb{N}^{\ast }\}.
\end{equation*}
In the assumptions of the original David-Journ\'{e} $T(1)$ theorem
the operators should also satisfy an extra condition the so called
weak boundedness. This condition is only used in the proof, as it
appears in \cite{Db}, to show that there exists some absolute
constant $C$, such that for all dyadic cubes $Q$
\begin{equation*}
\left| \int_{Q}T(\mathbf{1}_{Q})(x)dx\right| \leq C\left| Q\right|.
\end{equation*}
But since the operators $T_{\lambda}^{n}$ are canonically associated
with antisymmetric kernels the weak boundedness comes for free, see
e.g. \cite{Ch}.

Applying the $T(1)$ theorem we derive that every element of the
sequence $\{T_{\lambda }^{n}\}_{n\in \N}$ is bounded in $L^{2}(\mu
_{\lambda })$ with bounds not depending on $n$. This fact enables us
to extract some linear $L^{2}(\mu _{\lambda })$-bounded operator $T$
as a weak limit of some subsequence of $\{T_{\lambda }^{n}\}_{n\in
\N}$. Finally using the version of Cotlar's inequality, as it is
stated in (\cite{Db}, p.59), we get that there exists some constant
$C$ such that for all $f \in L^{2}(\mu_\lambda)$
\begin{equation*}
T_{\lambda }^{*}(f)(x)\leq
C(M_\lambda(Tf)(x)+{M_\lambda(|f|^{\sqrt{2}}(x))}^{\sqrt{2}}),
\end{equation*}
where $M_\lambda$ is the Hardy-Littlewood maximal operator related
to the measure $\mu_\lambda$. Therefore we conclude that $T_
{\lambda}^{*}$ is bounded in $L^{2}(\mu_\lambda)$.
\end{proof}

\section{Divergence of Principal Values}

\begin{thm}
Let $\lambda \in (0,1/3)$. For $\mu _{\lambda }$ almost every point
in $ E_{\lambda }$ the principal values of the singular integral
operator $T_{ \lambda}$ do not exist. \end{thm}

\begin{proof}Let $\lambda \in (0,1/3)$, we want to show that for $\mu _{\lambda }$ a.e $x\in
E_{\lambda }$ the limit
\begin{equation*}
\underset{\varepsilon \rightarrow 0}{\lim
}\left|\int_{\mathbb{R}^{2}\backslash B(x,\varepsilon )}K_{\lambda
}(x,y)d\mu _{\lambda }y\right|
\end{equation*}
does not exist.
\begin{figure}[b]
\centering
\includegraphics[scale = 0.3]{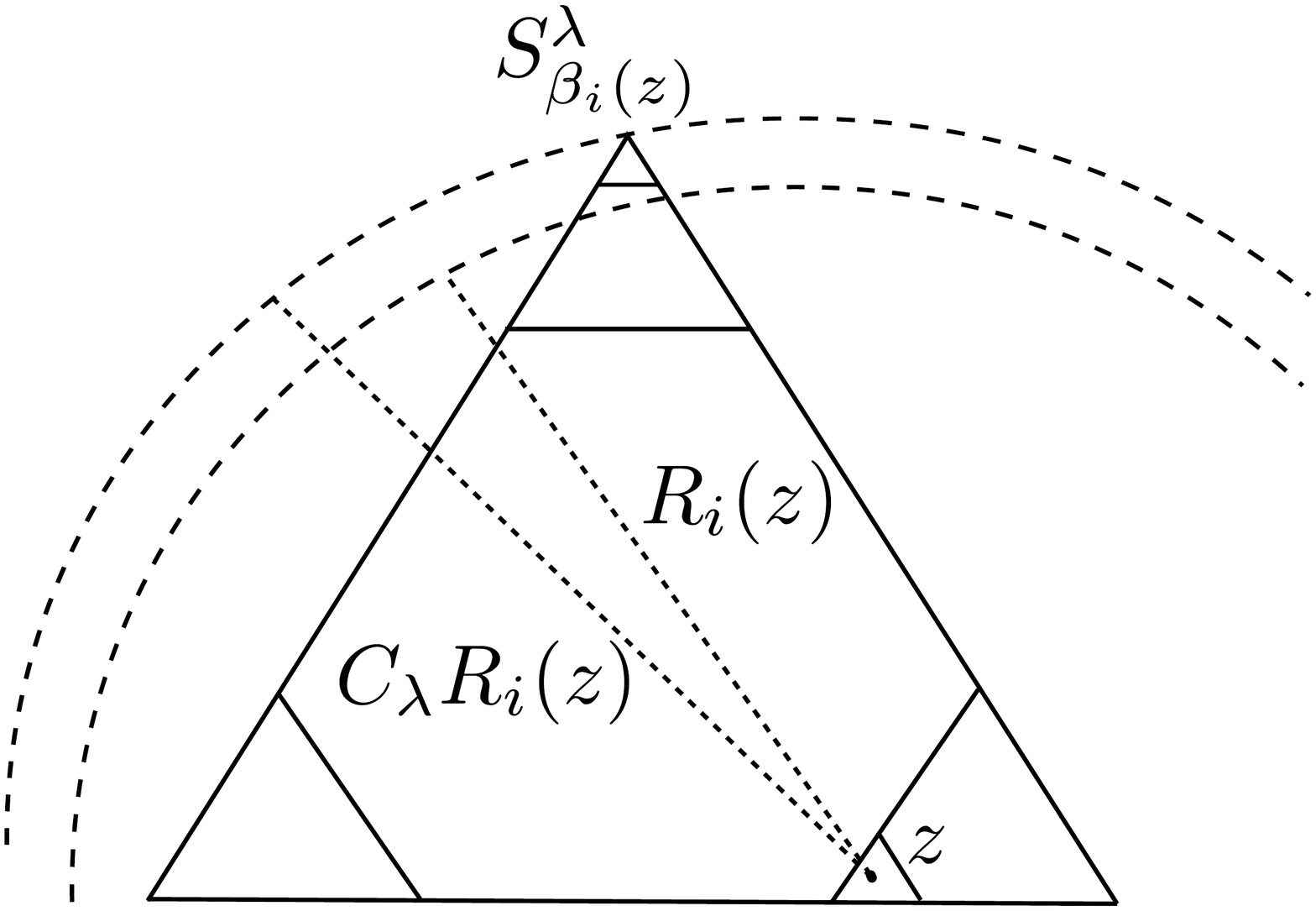}
\caption[]{}\label{pv}
\end{figure}
To every  $z\in E_{\lambda }$ assign naturally the code
$(z_{i})_{i\in \mathbb{N}}\in I^{\infty }$ such
 that $\{z\}=\underset{i\geq 1}{\bigcap
}S_{(z_{1},..,z_{i})}^{\lambda }$ and consider the set
\begin{equation*}
D_{\lambda }=\{z\in E_{\lambda }:z_{i}\neq z_{i+1}\text{ for
infinitely many }i^{\prime }s\}\text{ .}
\end{equation*}
The set $D_{\lambda }$ had full $\mu _\lambda$ measure as its
complement $E_{\lambda }\backslash D_{\lambda }$ is countable. In
fact the set $E_{\lambda }\backslash D_{\lambda }$ consists of the
vertices of every triangle $S_\alpha^\lambda$, $\alpha \in I^*$.

Notice that there exist some $C_{\lambda }>1$ and some $m_{\lambda
}\in \mathbb{N}$ such that for every $z\in D_{\lambda \text{ }}$ and
every $i\in \mathbb{N}^{\ast } $, satisfying $z_{i}\neq z_{i+1}$,
there exist  $ \beta _{i}(z) \in I^{i-1+m_{\lambda}}$ and positive
numbers $R_{i}(z)$ with the properties,
\begin{enumerate}
\item $\beta _{i}(z)=(z_{1},...,z_{i-1},\overset{m_{\lambda }\text{ times}}{%
\overbrace{y(z),..,y(z)}})$ where $y(z)\in
I\backslash\{z_{i},z_{i+1}\}$,
\item $R_{i}(z)\approx \lambda ^{i}$,
\item $(B(z,C_{\lambda }R_{i}(z))\backslash B(z,R_{i}(z)))\cap E_{\lambda
}=S_{\beta _{i}(z)}^{\lambda }$.
\end{enumerate}
See also Figure \ref{pv}. This geometric property of the sets
$E_\lambda$ forces the principal values of $T_{\lambda }$ to
diverge.

To see this, let some $x\in D_{\lambda }$ and denote $J_{x}=\{i\in
\mathbb{N}^{\ast }:x_{i}\neq x_{i+1}\}$. For all $i\in J_{x}$,
  \begin{eqnarray*}
  \lefteqn{\left| \int_{\mathbb{R}^{2}\backslash B(x,R_{i}(x))}K_{\lambda
}(x,y)d\mu _{\lambda }y-\int_{\mathbb{R}^{2}\backslash
B(x,C_{\lambda
}R_{i}(x))}K_{\lambda }(x,y)d\mu _{\lambda }y\right|} \\
  &=& \left| \int_{B(x,C_{\lambda }R_{i}(x))\backslash B(x,R_{i}(x))}\frac{%
\Omega _{\lambda }((x-y)\left| x-y\right| ^{-1})}{h_{\lambda
}(|x-y|)}d\mu
_{\lambda }y\right| \\
  &=&\left| \int_{S_{\beta _{i}(x)}^{\lambda }}\frac{\Omega _{\lambda
}((x-y)\left| x-y\right| ^{-1})}{h_{\lambda }(|x-y|)}d\mu _{\lambda
}y\right|
  \end{eqnarray*}
For all $x\in S_{a_{i}(x)}^{\lambda }$ and $y\in S_{\beta
_{i}(x)}^{\lambda } $, where $\alpha
_{i}(x)=(x_{1},...,x_{i},x_{i+1})$,
\begin{equation*}
(1-2\lambda )\lambda ^{i-1}\leq \left| x-y\right| \leq \lambda
^{i-1}
\end{equation*}
and
\begin{equation*}
 \Omega _{\lambda }\left( \frac{x-y}{\left| x-y\right|
}\right)=(-1)^{\varepsilon_i}
\end{equation*}
where $\varepsilon_i=1$ or $\varepsilon_i=-1$. Hence
\begin{eqnarray*}
\left| \int_{S_{\beta _{i}(x)}^{\lambda }}\frac{\Omega _{\lambda
}((x-y)\left| x-y\right| ^{-1})}{h_{\lambda }(|x-y|)}d\mu _{\lambda
}y\right|  &=&\int_{S_{\beta _{i}(x)}^{\lambda }}\frac{1}{h_{\lambda }(|x-y|)%
}d\mu _{\lambda }y \\
&= &\frac{\mu _{\lambda }(S_{\beta _{i}(x)}^{\lambda })}{(\lambda
^{i-1})^{d_{\lambda }}} \\
&=&\dfrac{(\lambda ^{i-1+m_{\lambda }})^{d_{\lambda }}}{(\lambda ^{i-1})^{d_{\lambda }}}%
=\lambda ^{m_{\lambda }d_{\lambda }}.
\end{eqnarray*}
As $R_{i}(x)\approx \lambda ^{i}\rightarrow 0$ we conclude that the
principal values of $ T_{\lambda }$ do not exist $\mu _\lambda $
a.e.
\end{proof}

\emph{Acknowledgements.} I am very grateful to my advisor, Professor
Pertti Mattila, for many ideas, discussions and suggestions during
the preparation of this note. I would also like to thank the referee
for useful comments and suggestions.

\vspace{1cm}
\begin{footnotesize}
{\sc Department of Mathematics and Statistics,
P.O. Box 68,  FI-00014 University of Helsinki, Finland,}\\
\emph{E-mail address:} \verb"vasileios.chousionis@helsinki.fi"
\end{footnotesize}

\end{document}